\documentclass[11pt,a4paper]{article}

\usepackage{url}
\usepackage{amsmath}
\usepackage{amssymb}
\usepackage{amsfonts}
\usepackage{euscript}
\usepackage{theorem}
\usepackage[all]{xypic}

\newtheorem{theorem}{Theorem}[section]
\newtheorem{proposition}[theorem]{Proposition}
\newtheorem{lemma}[theorem]{Lemma}
\newtheorem{corollary}[theorem]{Corollary}

\newenvironment{proof}{\medskip\emph{Proof.}}{\hfill$\Box$\medskip}

\newcommand{\Eff}{\mathsf{Eff}}
\newcommand{\Set}{\mathsf{Set}}

\newcommand{\pow}[1]{P(#1)}
\newcommand{\epow}[1]{\EuScript{P}(#1)}
\newcommand{\cpow}[1]{P_\omega(#1)}
\newcommand{\chains}[1]{\EuScript{C}(#1)}
\newcommand{\orth}[1]{\EuScript{D}(#1)}
\newcommand{\orthogonal}[1]{#1 \perp \nabla\two}

\newcommand{\pr}[1]{\varphi_#1}

\newcommand{\clx}{\mathsf{cl}_{\lnot\lnot}}
\newcommand{\cl}[1]{\clx#1}

\newcommand{\set}[1]{\{#1\}}
\newcommand{\such}{\mid}
\newcommand{\iimplies}{\mathbin{\Rightarrow}}
\newcommand{\Ex}[1]{\mathsf{E}_{#1}}

\newcommand{\tbigcap}{\bigcap\nolimits}
\newcommand{\tbigcup}{\bigcup\nolimits}

\newcommand{\defined}[1]{#1{\downarrow}}
\newcommand{\pair}[1]{\langle #1 \rangle}

\newcommand{\bound}{\mathsf{bound}}

\newcommand{\sem}[1]{[\![#1]\!]}

\newcommand{\NN}{\mathbb{N}}
\newcommand{\nno}{\mathsf{N}}

\newcommand{\two}{\mathsf{2}}
\newcommand{\one}{\mathsf{1}}

\newcommand{\obj}[1]{(|#1|,{=}_#1)}

\newcommand{\all}[3]{\forall\, #1 \,{\in}\, #2\,.\left(#3\right)}
\newcommand{\some}[3]{\exists\, #1 \,{\in}\, #2\,.\left(#3\right)}

\newcommand{\xall}[3]{\forall\, #1 \,{\in}\, #2\,.\,#3}
\newcommand{\xsome}[3]{\exists\, #1 \,{\in}\, #2\,.\,#3}

\newcommand{\xtlam}[3]{\lambda #1 \,{:}\, #2\,.\,#3}

\begin{document}

\title{On the Failure of Fixed-Point Theorems\\
  for Chain-complete Lattices\\
  in the Effective Topos}

\author{Andrej Bauer\\
  Faculty of Mathematics and Physics\\
  University of Ljubljana\\
  Ljubljana, Slovenia\\
  Email: \texttt{Andrej.Bauer@andrej.com}}

\maketitle

\begin{abstract} 
  In the effective topos there exists a chain-complete distributive
  lattice with a monotone and progressive endomap which does not have
  a fixed point. Consequently, the Bourbaki-Witt theorem and Tarski's
  fixed-point theorem for chain-complete lattices do not have
  constructive (topos-valid) proofs.
\end{abstract}

\section{Introduction}
\label{sec:introduction}

Fixed-point theorems state that maps have fixed points under certain
conditions. They are used prominently in denotational semantics, for
example to give meaning to recursive programs. In fact, it is hard to
overestimate their applicability and importance in mathematics in
general.

A constructive proof of a fixed-point theorem makes the theorem twice
as worthy because it yields an algorithm for computing a fixed point.
Indeed, many fixed-point theorems have constructive proofs, of which
we might mention Lawvere's fixed-point
theorem~\cite{lawvere06:_diagon_argum_and_cartes_closed_categ},
Tarski's fixed-point theorem for a monotone map on a complete
lattice~\cite{tarski55:_lattic_theor_fixpoin_theor_and_its_applic},
and Pataraia's generalization of it to directed-complete
posets~\cite{pataraia97:_const_proof_of_tarsk_fixed}. Two that have
defied constructive proofs are Tarski's theorem for chain-complete
posets and the Bourbaki-Witt
theorem~\cite{bourbaki49:_sur_le_theor_de_zorn,witt51:_beweis_zum_satz_von_m}
for progressive maps on chain-complete posets, see
Section~\ref{sec:discussion} for their precise statements.

I show that in the effective topos~\cite{HylandJ:efft} there is a
chain-complete distributive lattice with a monotone and progressive
endomap which does \emph{not} have a fixed point. An immediate
consequence of this is that both Tarski's theorem for chain-complete
posets and the Bourbaki-Witt theorem have no constructive
(topos-valid) proofs.

The outline of the argument is as follows. In the effective topos
$\Eff$ every chain is a quotient of a subobject of the natural
numbers, hence it has at most countably many global points.
Consequently, the (embedding into $\Eff$ of the) poset $\omega_1$ of
set-theoretic countable ordinals is chain-complete in the effective
topos, even though it is only countably complete in the topos of sets.
The successor function on $\omega_1$ is monotone, progressive, and
does not have a fixed point.
We work out the details of the preceding argument carefully in order
not to confuse external and internal notions of chain-completeness and
countability. We use~\cite{oosten08:_realiz} as a reference on the
effective topos. For the uninitiated, we have included a brief
overview of the effective topos in Appendix~\ref{sec:effective-topos}.

\section{Discrete objects in the effective topos}

An object in the effective topos is \emph{discrete}\footnote{The
  terminology is established and somewhat unfortunate, as it falsely
  suggests that a discrete object has decidable equality.} when it is
a quotient of a subobject of the natural numbers object~$\nno$. Such
objects were studied in~\cite{hyland90:_discr_objec_in_effec_topos},
where it is shown that~$X$ is discrete precisely when it is orthogonal
to~$\nabla \two$, by which we mean that the diagonal map $X \to
X^{\nabla\two}$ is an isomorphism. Here $\two = \set{0,1}$ is the
two-element set and $\nabla : \Set \to \Eff$ is the ``constant
objects'' functor, see Appendix~\ref{sec:functor-nabla}. In the
internal language of~$\Eff$ discreteness of~$X$ is expressed by the
statement
\begin{equation}
  \label{eq:discrete-global}
  \xall{f}{X^{\nabla\two}}{\xall{p}{\nabla\two}{f(p) = f(1)}},
\end{equation}
which says that every $f : \nabla\two \to X$ is constant. We are
interested in the object $\orth{X}$ of discrete subobjects of~$X$,
which we define in the internal language as
\begin{equation*}
  \orth{X} = \set{A \in \epow{X} \such \orthogonal{A}},
\end{equation*}
where $\epow{X}$ is the powerobject and $\orthogonal{A}$ is the
statement\footnote{We take care not to assume that a variable $A$
  ranging over a powerobject~$\epow{X}$ is an actual object in the
  topos, which is why~\eqref{eq:discrete-global}
  and~\eqref{eq:discrete-subset} differ slightly.}
\begin{equation}
  \label{eq:discrete-subset}
  \xall{f}{X^{\nabla\two}}{
    (\xall{p}{\nabla\two}{f(p) \in A}) \implies
    (\xall{p}{\nabla\two}{f(p) = f(1)})
  }.
\end{equation}
Let us explicitly compute $\orth{X}$ in case $X = \nabla S$ for a
set~$S$. The powerobject $\epow{\nabla S}$ is the set $\pow{\NN}^S$
with the non-standard equality predicate
\begin{equation*}
  [A =_{\epow{\nabla S}} B] = (A \iimplies B) \land (B \iimplies A).
\end{equation*}
The object $\orth{\nabla S}$ is the set $\pow{\NN}^S$ with
non-standard equality predicate
\begin{equation*}
  [A =_{\orth{\nabla S}} B] =
  (A \iimplies B) \land (B \iimplies A) \land D(A),
\end{equation*}
where $D : \pow{\NN}^S \to \pow{\NN}$ is a strict extensional relation
representing the predicate~\eqref{eq:discrete-subset}. To compute $D$
we recall how universal quantification over a constant object works.

Suppose $T$ is a set, $X$ is an object, and $\phi$ is a formula with
free variables $t$ and $x$ ranging over $\nabla T$ and $X$,
respectively, represented by the strict extensional relation $F : T
\times |X| \to \pow{\NN}$. Then the predicate $\xall{t}{\nabla
  T}{\phi}$ is represented by the strict extensional relation $|X| \to
\pow{\NN}$ defined by
\begin{equation*}
  x \mapsto {\bigcap\nolimits}_{t \in T} F(t, x).
\end{equation*}
When we apply this to the universal quantifiers
in~\eqref{eq:discrete-subset}, and use the fact that ${\nabla
  S}^{\nabla \two}$ is isomorphic to $\nabla{(S^\two)}$, we find after
a short calculation that
\begin{align*}
  D(A) &= \tbigcap_{f \in S^\two} A(f(0)) \cap A(f(1)) \iimplies
  [f(0) =_{\nabla S} f(1)] \\
  &= \tbigcap_{(x,y) \in S^\two} A(x) \cap A(y) \iimplies [x =_{\nabla
    S} y].
\end{align*}
We will need to know precisely when $D(A)$ is non-empty. If $x \neq y$
then $A(x) \cap A(y) \iimplies [x =_{\nabla S} y]$ is inhabited only
if $A(x) \cap A(y) = \emptyset$, because $x \neq y$ implies $[x
=_{\nabla S} y] = \emptyset$. Thus a necessary condition for $D(A)$ to
be non-empty is that $x \neq y$ implies $A(x) \cap A(y) = \emptyset$.
But this condition is also sufficient, since it implies that
\begin{multline*}
  D(A) =
  \tbigcap_{(x,y) \in S^\two}
  A(x) \cap A(y) \iimplies [x =_{\nabla S} y] = \\
  \left(
    \tbigcap_{x = y} A(x) \cap A(y) \iimplies [x =_{\nabla S} y]
  \right)
  \cap
  \left(
    \tbigcap_{x \neq y} A(x) \cap A(y) \iimplies [x =_{\nabla S} y]
  \right)
  = \\
  \left(
    \tbigcap_{x = y} A(x) \iimplies \NN
  \right)
  \cap
  \left(
    \tbigcap_{x \neq y} \emptyset \iimplies \emptyset
  \right) =
  \left(
    \tbigcap_{x = y} A(x) \iimplies \NN
  \right)
\end{multline*}
is non-empty because it contains at least (the G\"odel codes of) the
constant function $n \mapsto 0$.

Let $\clx : \epow{\nabla S} \to \nabla \pow{S}$ be the operator which
maps a subset to its double-negation closure:
\begin{equation*}
  \xymatrix@+2em{
    {\epow{\nabla S}}
    \ar[r]^{\cong}
    &
    {\Omega^{\nabla S}}
    \ar[r]^{{\lnot\lnot}^{\nabla S}}
    &
    {(\nabla \two)^{\nabla S}}
    \ar[r]^{\cong}
    &
    {\nabla \pow{S}}
  }
\end{equation*}
Let $\cpow{S}$ be the set of all countable subsets of a set~$S$.

\begin{proposition}
  \label{prop:clx-factor-cpow}
  For any set~$S$, the restriction of $\clx$ to $\orth{\nabla S}$
  factors through $\nabla \cpow{S}$:
  \begin{equation*}
    \xymatrix@+1em{
      {\orth{\nabla S}}
      \ar[r]^{i}
      \ar@{-->}[d]
      &
      {\epow{\nabla S}}
      \ar[d]^{\clx}
      \\
      {\nabla\cpow{S}}
      \ar[r]^{\nabla j}
      &
      {\nabla\pow{S}}
    }
  \end{equation*}
\end{proposition}

\begin{proof}
  In the diagram above $i$ and $j$ are inclusions $\orth{\nabla S}
  \subseteq \epow{\nabla S}$ and $\cpow{S} \subseteq \pow{S}$,
  respectively.
  Because $\nabla$ is right adjoint to the global points functor
  $\Gamma$, and $\Gamma \circ \nabla$ is naturally isomorphic to the
  identity, there is a unique $c : \Gamma(\orth{\nabla S}) \to
  \pow{S}$ such that $\clx \circ i$ is the composition of $\nabla c$ and
  the unit of the adjunction~$\eta$ at $\orth{\nabla S}$:
  \begin{equation*}
    \xymatrix@+1em{
      {\orth{\nabla S}}
      \ar[r]^{\eta}
      \ar[rd]_{\clx \circ i}
      &
      {\nabla\Gamma(\orth{\nabla S})}
      \ar[d]^{\nabla c}
      \\
      &
      {\nabla\pow{S}}
    }
  \end{equation*}
  It suffices to show that $c$ factors through $j$, since then $\clx
  \circ i = \nabla c \circ \eta$ factors through $\nabla j$.
  
  A global point $[A] : \one \to \orth{\nabla S}$ is represented by
  $A: S \to \pow{\NN}$ such that $D(A) \neq \emptyset$. Because $\clx$
  is composition with $\lnot\lnot$, we get
  \begin{equation*}
    c([A]) = \set{x \in S \such A(x) \neq \emptyset}.
  \end{equation*}
  Earlier we established that $D(A) \neq \emptyset$ implies $A(x) \cap
  A(y) = \emptyset$ whenever $x \neq y$. Therefore, for each $n \in
  \NN$ there is at most one $x \in A$ such that $n \in A(x)$, which
  means that there are at most countably many $x \in S$ for which
  $A(x) \neq \emptyset$. But then $c([A])$ is a countable subset
  of~$S$, which is what we wanted to prove.
\end{proof}

We shall need one more piece of knowledge about discrete objects.
Define the object $B = (\{0,1\}, {=_B})$ to have the equality
predicate
\begin{equation*}
  [x =_B y] =
  \begin{cases}
    \set{0} & \text{if $x = y = 0$,} \\
    \set{1} & \text{if $x = y = 1$,} \\
    \emptyset & \text{otherwise.}
  \end{cases}
\end{equation*}
The object $B$ is isomorphic to $\one + \one$. By the \emph{uniformity
  principle}~\cite[3.2.21]{oosten08:_realiz}, the following statement
is valid in the internal language of $\Eff$: for all $\phi \in
\epow{\nabla\two \times B}$, if
$\xall{p}{\nabla\two}{\xsome{d}{B}{\phi(p,d)}}$ then
$\xsome{d}{B}{\xall{p}{\nabla\two}{\phi(p,d)}}$. We require the
following equivalent form.

\begin{lemma}
  \label{lemma:nablatwo-forall-or}
  The following statement is valid in the internal language of $\Eff$:
  for all $\phi, \psi : \nabla\two \to \Omega$, if
  $\all{p}{\nabla\two}{\phi(p) \lor \psi(p)}$ then
  $\xall{p}{\nabla\two}{\phi(p)}$ or $\xall{p}{\nabla\two}{\psi(p)}$.
\end{lemma}

\begin{proof}
  We argue in the internal language of~$\Eff$. If
  $\all{p}{\nabla\two}{\phi(p) \lor \psi(p)}$ then
  \begin{equation*}
    \xall{p}{\nabla\two}{\some{d}{\two}{
        (d = 0 \land \phi(p)) \lor
        (d = 1 \land \psi(p)})}.
  \end{equation*}
  To see this, take $d = 0$ if $\phi(p)$ holds and $d = 1$ if
  $\psi(p)$ holds. By the uniformity principle
  \begin{equation*}
    \xsome{d}{\two}{\all{p}{\nabla\two}{
        (d = 0 \land \phi(p)) \lor
        (d = 1 \land \psi(p)})}.
  \end{equation*}
  Consider such $d \in \two$. If $d = 0$ then
  $\xall{p}{\nabla\two}{\phi(p)}$, and if $d = 1$ we obtain
  $\xall{p}{\nabla\two}{\psi(p)}$.
\end{proof}

\section{Posets and Chains in the Effective Topos}
\label{sec:chains}

In this section we work entirely in the internal language of the
effective topos. First we recall several standard order-theoretic
notions. A \emph{poset} $(L, {\leq})$ is an object $L$ with a relation
$\leq$ which is reflexive, transitive, and antisymmetric. A
\emph{lattice} $(L, {\leq}, {\land}, {\lor})$ is a poset in which
every pair of elements $x, y \in L$ has a greatest lower bound $x
\land y$, and least upper bound $x \lor y$. Note that a lattice need
not have the smallest and the greatest element. A lattice is
\emph{distributive} if $\land$ and $\lor$ satisfy the distributivity
laws $(x \land y) \lor z = (x \lor z) \land (y \lor z)$ and $(x \lor
y) \land z = (x \land z) \lor (y \land z)$. An endomap $f : L \to L$
on a poset $(L, {\leq})$ is \emph{monotone} when
\begin{equation*}
  \all{x,y}{L}{x \leq y \implies f(x) \leq f(y)},
\end{equation*}
and \emph{progressive} when $\xall{x}{L}{x \leq f(x)}$.

For $x \in L$ and $A \in \epow{L}$ define $\bound(x,A)$ to be
the relation
\begin{equation*}
  \bound(x,A) \iff \all{y}{L}{y \in A \implies y \leq x}.
\end{equation*}
We say that $z \in L$ is the \emph{supremum} of $A \in \epow{L}$ when
\begin{equation*}
  \bound(z,A) \land
  \all{y}{L}{\bound(y,A) \implies y \leq z}.
\end{equation*}

\begin{lemma}
  \label{lemma:sup-notnot}
  Suppose $(L, {\leq})$ is a poset with a $\lnot\lnot$-stable order.
  For all $A \in \epow{L}$ and $x \in L$, if $x$ is the supremum of
  $\cl{A}$ then $x$ is the supremum of~$A$.
\end{lemma}

\begin{proof}
  By definition of $\clx$, $y \in \cl{A}$ is equivalent to
  $\lnot\lnot(y \in A)$. If $\leq$ is $\lnot\lnot$-stable then
  \begin{align*}
    \bound(x, \cl{A}) &\iff
    \all{y}{L}{\lnot\lnot(y \in A) \implies y \leq x} \\
    &\iff
    \all{y}{L}{y \in A \implies \lnot\lnot(y \leq x)} \\
    &\iff
    \all{y}{L}{y \in A \implies y \leq x} \\
    & \iff
    \bound(x, A).
  \end{align*}
  Because $\cl{A}$ and $A$ have the same upper bounds, if $x$ is the
  supremum of one of them then it is the supremum of the other as
  well.
\end{proof}

By a \emph{chain} in a poset $(L, {\leq})$ we mean $C \in \epow{L}$
such that
\begin{equation*}
  \all{x,y}{L}{x \in C \land y \in C \implies x \leq y \lor y \leq x}.
\end{equation*}
The \emph{object of chains in~$L$} is defined as
\begin{equation*}
  \chains{L} = \set{C \in \pow{L} \such
    \all{x,y}{L}{x \in C \land y \in C \implies x \leq y
      \lor y \leq x}
  }.
\end{equation*}

\begin{proposition}
  \label{prop:chain-discrete}
  Every chain is discrete, i.e., $\chains{L} \subseteq \orth{L}$.
\end{proposition}

\begin{proof}
  Consider any $C \in \chains{L}$ and $f : \nabla\two \to L$ such that
  $\xall{p}{\nabla\two}{f(p) \in C}$. We need to show that $f$ is
  constant. Because~$C$ is a chain we have
  \begin{equation*}
    \all{p,q}{\nabla\two}{f(p) \leq f(q) \lor f(q) \leq f(p)}.
  \end{equation*}
  By a double application of Lemma~\ref{lemma:nablatwo-forall-or} we
  obtain
  \begin{equation*}
    (\xall{p,q}{\nabla\two}{f(p) \leq f(q)}) \lor
    (\xall{p,q}{\nabla\two}{f(q) \leq f(p)}).
  \end{equation*}
  Because $\leq$ is antisymmetric, both disjuncts imply $f(p) = f(q)$
  for all $p, q \in \nabla\two$, as required.
\end{proof}

\section{The poset $\nabla\omega_1$}
\label{sec:nabla-omega1}

Let $(\omega_1, {\preceq})$ be the distributive lattice of countable
ordinals in $\Set$. This is not a chain-complete poset, but it is
complete with respect to countable subsets. Let $\sup :
\cpow{\omega_1} \to \omega_1$ be the supremum operator which maps a
countable subset $A \subseteq \omega_1$ to its supremum.

The object $\nabla\omega_1$, ordered by $\nabla{\preceq}$, is a
distributive lattice in $\Eff$. One way to see this is to observe that
$\nabla$ preserves finite products, therefore it maps models of the
equational theory of distributive lattices to models of the same
theory. Moreover, $\nabla$ also preserves the statement
\begin{equation*}
  \xall{A}{\cpow{S}}{\text{``$\sup(A)$ is the supremum of~$A$''}}
\end{equation*}
because the statement is expressed in the negative fragment of logic
($\land$, $\implies$, $\forall$), which is preserved by~$\nabla$.

\begin{proposition}
  The poset $\nabla\omega_1$ is chain-complete in $\Eff$.
\end{proposition}

\begin{proof}
  We claim that the supremum operator $\chains{\nabla\omega_1} \to
  \nabla\omega_1$ is the composition
  \begin{equation*}
    \xymatrix@+2.5em{
      {\chains{\nabla\omega_1}}
      \ar[r]^{\subseteq}
      &
      {\orth{\nabla\omega_1}}
      \ar[r]^{\clx}
      &
      \nabla(\cpow{\omega_1})
      \ar[r]^{\nabla\sup}
      &
      \nabla\omega_1
    }
  \end{equation*}
  The arrow marked by $\subseteq$ comes from 
  Lemma~\ref{prop:chain-discrete}, while the one marked as $\clx$ is
  the factorization $\orth{\nabla\omega_1} \to \nabla \cpow{\omega_1}$
  from Proposition~\ref{prop:clx-factor-cpow}.

  We argue in the internal language of~$\Eff$. Consider a chain $C \in
  \chains{\nabla\omega_1}$. Then $\cl{C} \in \nabla\cpow{\omega_1}$,
  therefore $(\nabla \sup)(\cl{C})$ is the supremum of $\cl{C}$. But
  since the order $\nabla{\leq}$ on $\nabla\omega_1$ is
  $\lnot\lnot$-stable it is also the supremum of $C$ by
  Lemma~\ref{lemma:sup-notnot}.
\end{proof}

\begin{corollary}
  \label{cor:counter-example}
  In the effective topos, there is a chain-complete distributive
  lattice with a monotone and progressive endomap which does
  \emph{not} have a fixed point.
\end{corollary}

\begin{proof}
  The successor map $\mathsf{succ} : \omega_1 \to \omega_1$ is
  monotone, progressive, and does not have a fixed point. The functor
  $\nabla$ preserves these properties because they are all expressed
  in the negative fragment. Therefore, in the effective topos
  $\nabla\omega_1$ is a chain-complete distributive lattice and
  $\nabla\mathsf{succ}$ is monotone, progressive and does not have a
  fixed point.
\end{proof}

\section{Discussion}
\label{sec:discussion}

An immediate consequence of Corollary~\ref{cor:counter-example} is
that the following theorems \emph{cannot} be proved constructively,
i.e., in higher-order intuitionistic logic:
\begin{enumerate}
\item Tarski's
  Theorem~\cite{tarski55:_lattic_theor_fixpoin_theor_and_its_applic}
  for chain-complete lattices: a monotone map on a chain-complete
  lattice has a fixed point.
\item Bourbaki-Witt
  theorem~\cite{bourbaki49:_sur_le_theor_de_zorn,witt51:_beweis_zum_satz_von_m}:
  a progressive map on a chain-complete poset has a fixed point above
  every point.
\end{enumerate}
The theorems cannot be proved even if we assume Dependent Choice
because it is valid in the effective topos.

Dito Pataraia~\cite{pataraia97:_const_proof_of_tarsk_fixed} proved
constructively Tarski's fixed-point theorem for dcpos. A natural
question is whether perhaps the Bourbaki-Witt theorem can also be
proved constructively for dcpos. The following observation by France
Dacar~\cite{dacar08:_suprem_of_famil_of_closur_operat} shows that this
is not possible because the Bourbaki-Witt theorems for chain-complete
posets and dcpos are constructively equivalent.

\begin{theorem}[France Dacar]
  The following are constructively equivalent:
  \begin{enumerate}
  \item
    \label{item:ccpo}
    Every progressive map on a chain-complete inhabited poset has
    a fixed point.
  \item
    \label{item:dcpo}
    Every progressive map on a directed-complete inhabited poset
    has a fixed point.
  \end{enumerate}
\end{theorem}

\begin{proof}
  For this theorem we require chains to be inhabited.\footnote{So far
    we could work with possibly uninhabited chains because the poset
    of interest $\nabla\omega_1$ has a least element.} The direction
  from chain-complete posets to directed-complete ones is trivial
  because every directed-complete poset is chain-complete.
  To prove the converse, suppose~\eqref{item:dcpo} holds and let $(P,
  {\leq})$ be a chain-complete inhabited poset with a progressive map
  $f : P \to P$. The set $C$ of inhabited chains in~$P$, ordered by
  inclusion, is inhabited and closed under directed unions, therefore
  it is a dcpo. Define the map $F : C \to C$ by $F(A) = A \cup
  {f(\sup(A))}$. This is a progressive map on~$C$, therefore
  by~\eqref{item:dcpo} it has a fixed point~$B$. Now $f(\sup(B)) \in
  B$ and hence $f(\sup{B}) \leq \sup{B}$, which means that $\sup(B)$
  is a fixed point of~$f$.
\end{proof}

In constructive mathematics the tradition is not to despair when a
classical theorem turns out to be unprovable, but rather to find a
constructively acceptable formulation and prove it. What that might be
in the present case remains to be seen.

Finally, let us remark that Giuseppe Rosolini~\cite{RosoliniG:modpti}
showed that in a certain realizability model for the intuitionistic
Zermelo-Fraenkel set theory IZF the trichotomous ordinals are
precisely the discrete ordinals which are at most subcountable. Such
ordinals therefore form a set in the model, rather than a class. From
this it follows that the Bourbaki-Witt theorem fails in the model
because the successor map is progressive and has no fixed point.
However, Tarski's theorem for chain-complete posets is not invalidated
because the successor map is \emph{not} monotone in the model. Both
proofs, Rosolini's and the present one clearly use discrete objects in
a similar way.

After this work was presented at the Mathematical Foundations of
Programming Semantics~25 in Oxford, the question arose whether the
Bourbaki-Witt theorem is valid in sheaf toposes. I have recently been
told by Peter Lumsdaine that this is indeed the case because the
inverse image part of a geometric morphism $\mathcal{E} \to
\mathcal{F}$ transfers the Bourbaki-Witt theorem from~$\mathcal{F}$
to~$\mathcal{E}$. Thus, in order to establish the Bourbaki-Witt
theorem in a sheaf topos $\mathcal{E}$ (or in fact any cocomplete
topos), we consider the geometric morphism $\mathcal{E} \to \Set$
whose direct image is the global global sections functor.

\paragraph{Acknowledgment.} I thank France Dacar for inspiration and
many useful bits of knowledge. I also thank Benno van den Berg and
Giuseppe Rosolini for explaining that~\cite{RosoliniG:modpti} implies
constructive failure of the Bourbaki-Witt theorem. I am indebted to
Gisuppere Rosolini who kindly acted as my surrogate at the MFPS~25
conference, which I was unable to attend.



\appendix

\section{The Effective Topos}
\label{sec:effective-topos}

We rely on~\cite{oosten08:_realiz} as a reference on the effective
topos and give only a quick overview of the basic constructions here.

\subsection{Definition of the effective topos}

Recall that a \emph{non-standard} predicate on a set~$X$ is a map $A :
X \to \pow{\NN}$, where we think of $A(x)$ as the set of realizers
(G\"odel codes of programs) which witness the fact that $x$ has the
property $A$. The non-standard predicates on~$X$ form a Heyting
prealgebra $\pow{\NN}^X$ with the partial order
\begin{equation*}
  A \leq B \iff \xsome{n}{\NN}{\xall{x}{X}{\xall{m}{A(x)}{
        \defined{\pr{n}(m)} \land \pr{n}(m) \in B(x)}}},
\end{equation*}
where $\pr{n}$ is the $n$-th partial recursive function and
$\defined{\pr{n}(m)}$ means that $\pr{n}(m)$ is defined. In words, $A$
entails $B$ if there is a program that translates realizers for $A(x)$
to realizers for $B(x)$, uniformly in $x$. Predicates $A$ and $B$ are
\emph{equivalent} when $A \leq B$ and $A \leq B$. If we quotient
$\pow{\NN}^X$ by the equivalence we obtain an honest Heyting algebra, but
we do not do that.

Let $\pair{{-},{-}}$ be a computable pairing function on the natural
numbers~$\NN$, e.g., $\pair{m,n} = 2^m (2 n + 1)$. The Heyting
prealgebra structure of $\pow{\NN}^X$ is as follows:
\begin{align}
  \label{eq:heyting-preaglebra}
  \top(x) &= \NN \\
  \bot(x) &= \emptyset \notag \\
  (A \land B)(x) &=
  \set{\pair{m,n} \such m \in A(x) \land n \in B(x)} \notag \\
  (A \lor B)(x) &=
  \set{\pair{0,n} \such n \in A(x)} \cup
  \set{\pair{1,n} \such n \in B(x)} \notag \\
  (A \iimplies B)(x) &=
  \set{n \in \NN \such
    \xall{m}{A(x)}{\defined{\pr{n}(m)} \land \pr{n}(m) \in B(x)}
  }. \notag
\end{align}
We say that a non-standard predicate $A$ is \emph{valid} if $\top \leq
A$, in which case we write $\models A$. The condition $\top \leq A$ is
equivalent to requiring that $\bigcap_{x \in X} A(x)$ contains at
least one number. Often a non-standard predicate is given as a map $x
\mapsto \phi(x)$ where $\phi$ is an expression with a free
variable~$x$. In this case we abuse notation and write $\models
\phi(x)$ instead of $\models \xtlam{x}{X}{\phi(x)}$. In other words,
free variables are to be implicitly abstracted over.

An object $X = \obj{X}$ in the effective topos is a set $|X|$ with a
non-standard \emph{equality predicate} ${=_X} : |X| \times |X| \to
\pow{\NN}$, which is required to be symmetric and transitive (where we
write $[x =_X y]$ instead of $x =_X y$ for better readability):
\begin{align*}
  & \models [x =_X y] \iimplies [y =_X x], \tag{symmetric} \\
  & \models [x =_X y] \land [y =_X z] \iimplies [x =_X z].
  \tag{transitive}
\end{align*}
Usually we write $\Ex{X}(x)$ for $[x =_X x]$. Think of $\Ex{X}$ as an
``existence predicate'', and $\Ex{X}(x)$ as the set of realizers which
witness the fact that $x$ exists.

In the effective topos a morphism $F : X \to Y$ is represented by a
non-standard functional relation $F : X \times Y \to \pow{\NN}$. More
precisely, we require that
\begin{align*}
  & \models F(x,y) \iimplies \Ex{X}(x) \land \Ex{Y}(y)
  \tag{strict} \\
  & \models [x' =_X x] \land F(x,y) \land [y =_Y y'] \iimplies F(x', y')
  \tag{extensional} \\
  & \models F(x,y) \land F(x,y') \iimplies [y =_X y'] 
  \tag{single-valued} \\
  & 
  \models \Ex{X}(x) \iimplies \tbigcup_{y \in Y} \Ex{Y}(y) \land F(x, y).
  \tag{total}
\end{align*}
Two such functional relations $F, F'$ represent the same morphism when
$F \leq F'$ and $F' \leq F$ in the Heyting prealgebra $\pow{\NN}^{X
  \times Y}$. Composition of $F : X \to Y$ and $G : Y \to Z$ is the
functional relation $G \circ F$ given by
\begin{equation*}
  (G \circ F)(x, z) = \tbigcup_{y \in Y} {F(x, y) \land G(y, z)}.
\end{equation*}
The identity morphism $I : X \to X$ is represented by the relation
$I(x,y) = [x =_X y]$.

The category $\Eff$ is a topos. Let us give a description of
powerobjects. If $X$ is an object then the \emph{powerobject}
$\epow{X}$ is the set $\pow{\NN}^{|X|}$ with the non-standard equality
predicate
\begin{multline*}
  [A =_{\epow{X}} B] =
  (A \iimplies B) \land (B \iimplies A) \land {} \\
    \left(\tbigcap_{x \in |X|} A(x) \iimplies \Ex{X}(x) \right) \land
    \left(\tbigcap_{x,y \in |X|} A(x) \land [x =_X y] \iimplies A(y)
    \right).
\end{multline*}
The complicated part in the second line says that $A$ is strict and
extensional. If $x$ and $y$ are variables of type~$X$ and~$\epow{X}$,
respectively, then the atomic predicate $x \in y$ is represented by
the strict extensional predicate $E : |X| \times \pow{\NN}^{|X|} \to
\pow{\NN}$ defined by $E(u,A) = \Ex{X}(u) \land \Ex{\epow{X}}(A) \land
A(u)$.

\subsection{Interpretation of first-order logic in $\Eff$}

The effective topos supports an interpretation of intuitionistic
first-order logic, which we outline in this section.
Each subobject of an object $X = \obj{X}$ is represented by a
\emph{strict extensional predicate}, which is a non-standard predicate
$A : |X| \to \pow{\NN}$ that satisfies:
\begin{align*}
  & \models A(x) \iimplies \Ex{X}(x), \tag{strict} \\
  & \models A(x) \land [x =_X x'] \iimplies A(x'). \tag{extensional}
\end{align*}
Such a predicate represents the subobject determined by the mono $I :
Y \to X$ where $|Y| = |X|$, $[x =_Y y] = [x =_X y] \land A(x)$, and
$I(x,y) = [x =_Y y]$. Strict predicates represent the same
subobject precisely when they are equivalent as elements of the
Heyting prealgebra $\pow{\NN}^X$.

The interpretation of first-order logic with equality in $\Eff$ may be
expressed in terms of strict extensional predicates and non-standard
equality predicates. Suppose $\phi$ is a formula with a free
variable~$x$ ranging over an object~$X$.\footnote{In the general case
  $\phi$ may contain free variables $x_1, \ldots, x_n$ ranging over
  objects $X_1, \ldots, X_n$, respectively. Such a $\phi$ is
  interpreted as a subobject of $X_1 \times \cdots \times X_n$. It is
  easy to work out the details once you have seen the case of a single
  variable.} The interpretation of $\phi$ is the subobject of $X$
represented by the non-standard predicate $\sem{\phi} : |X| \to
\pow{\NN}$, defined inductively on the structure of~$\phi$ as follows.
The propositional connectives are interpreted by the Heyting
prealgebra structure of non-standard predicates,
cf.~\eqref{eq:heyting-preaglebra}:
\begin{align*}
  \sem{\top} &= \top \\
  \sem{\bot} &= \bot \\
  \sem{\theta \land \psi} &= \sem{\theta} \land \sem{\psi} \\
  \sem{\theta \lor \psi} &= \sem{\theta} \lor \sem{\psi} \\
  \sem{\theta \iimplies \psi} &= \sem{\theta} \iimplies \sem{\psi}.
\end{align*}
Suppose $\psi$ is a formula with free variables $x$ of type~$X$ and
$y$ of type~$Y$, and let $A = \sem{\psi} : |X| \times |Y| \to
\pow{\NN}$ be a strict extensional predicate which interprets $\psi$.
Then the interpretation of the quantifiers is:
\begin{align}
  \sem{\xsome{x}{X}{\psi}}(y) &=
  \tbigcup_{x \in |X|} \Ex{X}(x) \land A(x, y),
  \label{eq:interpret-quant}
  \\
  \sem{\xall{x}{X}{\psi}}(y) &=
  \tbigcap_{x \in |X|} \Ex{X}(x) \iimplies A(x, y).
  \notag
\end{align}
Suppose $f, g : X \to Y$ are morphisms represented by functional
relations $F, G : |X| \times |Y| \to \pow{\NN}$, respectively. The
atomic formula $f = g$, where $x$ is a variable of type~$X$, is
interpreted as the subobject of $X$ represented by the non-standard
predicate $\sem{f = g} : |X| \to \pow{\NN}$, defined by
\begin{equation*}
  \sem{f = g}(x) = \tbigcup_{y \in |Y|} F(x,y) \land G(x,y).
\end{equation*}
If other atomic predicates appear in a formula, their interpretation
must be given in terms of corresponding strict extensional predicates.

\subsection{The functor $\nabla : \Set \to \Eff$}
\label{sec:functor-nabla}

The topos of sets~$\Set$ is (equivalent to) the topos of sheaves for
the $\lnot\lnot$-topology on~$\Eff$. The direct image part of the
inclusion $\Set \to \Eff$ is the functor $\nabla : \Set \to \Eff$
which maps a set $S$ to the object $\nabla S = (S, {=_{\nabla S}})$
where
\begin{equation*}
  [x =_{\nabla S} y] =
  \begin{cases}
    \NN & \text{if $x = y$,} \\
    \emptyset & \text{if $x \neq y$.}
  \end{cases}
\end{equation*}
A map $f : S \to T$ is mapped to the morphism $\nabla f : \nabla S \to
\nabla T$ represented by the functional relation
\begin{equation*}
  (\nabla f)(x,y) = [f(x) =_{\nabla T} y] \;.
\end{equation*}
The inverse image part is the global sections functor $\Gamma : \Eff
\to \Set$, defined as $\Gamma(X) = \Eff(\one, X)$. Concretely, a
global point $\one \to X$ is represented by an element $x \in |X|$
such that $\Ex{X}(x) \neq \emptyset$. Two such $x, y \in |X|$
represent the same global point when $[x =_X y] \neq \emptyset$.

If $S$ is a set then every element of $\nabla S$ exists uniformly, in
the sense that $\Ex{S}(x) = \NN$. Every map $S \to \pow{\NN}$ is
strict and extensional with respect to $=_{\nabla S}$. These two
observations allow us to simplify calculations involving $\nabla S$.
For example, the powerobject $\epow{\nabla S}$ is the set
$\pow{\NN}^S$ with the equality predicate simplified to $[A
=_{\epow{\nabla S}} B] = (A \iimplies B) \land (B \iimplies A)$.
Similarly, the interpretation~\eqref{eq:interpret-quant} of
existential and universal quantifiers simplifies to
\begin{align*}
  \sem{\xsome{x}{\nabla S}{\psi}}(y) &=
  \tbigcup_{x \in S} A(x, y), \\
  \sem{\xall{x}{\nabla S}{\psi}}(y) &=
  \tbigcap_{x \in S} A(x, y).
\end{align*}

\end{document}